\documentclass[11pt]{amsart}
\usepackage{amssymb}
\usepackage[all]{xy}
\usepackage{graphicx}
\newcommand{\PP}{\mathbb P}

\newcommand\CC{{\mathbb C}}

\newcommand\KK{{\mathbb K}}

\newcommand{\mZ}{\mathcal Z} 
\newcommand{\mT}{\mathcal T}

\newcommand{\TLL}{\hat{L}}

\DeclareMathOperator{\sing}{\operatorname{sing}}

\numberwithin{equation}{section}
\begin{document}

\title{On quartics with lines of the second kind}

\author{S\l awomir Rams}
\address{Institute of Mathematics, Jagiellonian University, 
ul. {\L}ojasiewicza 6,  30-348 Krak\'ow, Poland}
\address{Current address: 
Institut f\"ur Algebraische Geometrie, Leibniz Universit\"at
  Hannover, Welfengarten 1, 30167 Hannover, Germany} 
\email{slawomir.rams@uj.edu.pl}

\author{Matthias Sch\"utt}
\address{Institut f\"ur Algebraische Geometrie, Leibniz Universit\"at
  Hannover, Welfengarten 1, 30167 Hannover, Germany}
\email{schuett@math.uni-hannover.de}

\thanks{Funding by ERC StG 279723 (SURFARI) and NCN grant no. N N201 608040 (S. Rams) is gratefully acknowledged.}

\date{March 6, 2013}

\begin{abstract}
We study the geometry of quartic surfaces in 
$\PP^{3}$ that contain a line  of the second kind
over algebraically closed fields of characteristic different from 2,3.
In particular, we 
correct Segre's claims made for the complex case in 1943.
\end{abstract}

\maketitle

\newcommand{\XXd}{X_{d}}
\newcommand{\XXf}{X}
\newcommand{\XXp}{X_{5}}
\newcommand{\Ruledeight}{S_{8}}
\newcommand{\Ruledfive}{S_{5}}
\newcommand{\Ruledfourf}{S}
\newcommand{\Ruledfour}{S_{4}}
\newcommand{\Ruledtwo}{S_{2}}
\newcommand{\DivisorRest}{D}
\newcommand{\Pl}{\Pi}
\newcommand{\reg}{\operatorname{reg}}
\theoremstyle{remark}
\newtheorem{obs}{Observation}[section]
\newtheorem{rem}[obs]{Remark}
\newtheorem{example}[obs]{Example}
\newtheorem{claim}[obs]{Claim}
\theoremstyle{definition}
\newtheorem{Definition}[obs]{Definition}
\theoremstyle{plain}
\newtheorem{prop}[obs]{Proposition}
\newtheorem{theo}[obs]{Theorem}
\newtheorem{lemm}[obs]{Lemma}
\newtheorem{constr}[obs]{Construction}
\newtheorem*{conj}{Conjecture}
\newtheorem{cor}[obs]{Corollary}

\newcommand{\ux}{\underline{x}}
\newcommand{\ud}{\underline{d}}
\newcommand{\ue}{\underline{e}}
\newcommand{\mmS}{{\mathcal S}}
\newcommand{\mmP}{{\mathcal P}}
\newcommand{\nlines}{\mbox{\texttt l}(\XXp)}
\newcommand{\ii}{\operatorname{i}}

\newcommand{\nonlinflec}{{\mathcal D}}
\newcommand{\linflec}{{\mathcal L}}
\newcommand{\flec}{{\mathcal F}}

\section{Introduction} \label{sect-introduction}

The main aim of this note is to study the geometry of certain quartics in $\PP^3$ that contain
a line. 
Our investigations concern not only the field of complex numbers $\CC$,
but any algebraically closed field $\KK$ of characteristic  $\neq$ $2$, $3$. 
Specifically, let $\XXf \subset \PP^3_{\KK}$  be a smooth quartic surface that contains a line  $\ell$. 
The linear 
system $|{\mathcal O}_{\XXf}(1) - \ell|$ endows the surface in question with an elliptic  
fibration
\begin{equation} \label{eq:fibration}
\pi \, : \, \XXf \to\PP^1 .
\end{equation}
The geometry of \eqref{eq:fibration} can in principle be used to determine all lines on $X$.
Notably, any line on $X$ meeting $\ell$ 
occurs as component of a singular fiber of \eqref{eq:fibration}.
Euler number considerations yield 
that $\ell$ meets at most $24$ other lines on $X$.
Following Segre \cite{Segre} this bound can be improved 
by noticing that the intersection point of $\ell$ with any other line on $X$
lies in the closure  of the locus of  inflection points of smooth fibers of the fibration
\eqref{eq:fibration}. 
Thus elimination theory 
gives the (sharp) upper bound of 18 other lines on $X$
possibly met by $\ell$
unless $\ell$ is contained in the flex locus (cf.~\cite[p.88]{Segre},~\cite[Lem.~5.2]{ramsschuett}).
Therefore  the lines satisfying the latter condition are crucial for a good understanding of 
possible configurations of lines on smooth quartics.
This motivated B.~Segre to formulate 
the following definition (cf. \cite[p.~87]{Segre}) that  played  central role in his attempt
to show that a smooth complex
 quartic contains at most 64 lines (as proved in \cite[Thm. 1.1]{ramsschuett}).  
 
\begin{Definition}[Lines of the second kind]
The line  $\ell$ is of the second kind iff it is contained in the closure  of the flex locus of the smooth fibers of the fibration
\eqref{eq:fibration}. 
\end{Definition}


 The starting point of this paper is Segre's observation that every smooth  quartic obtained by perturbing
the equation of a ruled quartic $\Ruledfourf$ with  a product of 4 planes
that vanish along  a directrix $\ell$ of $\Ruledfourf$, i.e.  
\begin{equation} \label{eq:segint}
\Ruledfourf + \TLL_1 \cdot \ldots \cdot \TLL_4,  \mbox{ where the planes }  \TLL_1, \ldots, \TLL_4 \mbox{ meet along the directrix } \ell , 
\end{equation}
contains $\ell$ as a line of the second kind (see Construction~\ref{obs-segretrick}).
Moreover, Segre  claimed that any quartic  with a line of the second kind is given by \eqref{eq:segint}.
The latter statement was used in \cite{Segre} to 
deduce various properties of quartics with lines of the second kind (see Sect.~\ref{sect-segreclaims} for details).  
However, in full generality Segre's claims are wrong.   
Notably, in \cite{ramsschuett} we exhibited  a family $\mZ$ of quartics
 that crucially contradicted certain bounds stated in  
\cite{Segre} (see Sect.~\ref{s:corr}).
Since our main aim was  to give a correct characteristic-free proof of the 
bound on the number of lines, 
the question why Segre's idea  does not work in general was not adressed in 
\cite{ramsschuett}.
 
In this paper, we study the geometry of smooth quartics with a line of the second kind in detail.
Following an idea from \cite{ramsschuett}
we show that there are 3 distinct families of smooth quartics with such  lines, 
distinguished by the ramification type $R$ of the morphism $\pi|_\ell$,
the restriction of \eqref{eq:fibration} to $\ell$.
Here $\pi|_\ell$ has either 2 ramification points (type $R=2^2$),
3 ramification points (type $R=2,1^2$), or 4 ramification points (type $R=1^4$). The following condition on the fibration \eqref{eq:fibration} will arise naturally in this paper:
\begin{equation} \label{eq:segcond}
\mbox{ If a fiber  } F \mbox{ of \eqref{eq:fibration} meets } \ell  \mbox{ in exactly one point, then } F \mbox{ has type } IV . 
\end{equation}

The main results of this paper are summarised in the following theorem,
see Sect.~\ref{s:thm} for a collection of the ingredients of the proof and 
Sect.~\ref{s:corr2},~\ref{s:corr} for the definition of the families $\mT$, $\mZ$.

\begin{theo}
\label{thm}
Let $\XXf$
 be a  smooth quartic with a line $\ell$  of the second kind of ramification type $R$.
\begin{itemize}
\item[(a)]
The surface $\XXf$ can be obtained by Segre's construction \eqref{eq:segint}
if and only if the fibration  \eqref{eq:fibration} satifies the condition \eqref{eq:segcond}.  
\item[(b)]
If $R=2,1^2$
(resp. $R=2^2$), then $\XXf$ is projectively equivalent to a member of the family $\mT$ (resp. $\mZ$).
\end{itemize}
In particular, every quartic  with a line of the second kind of ramification type $R=1^4$ can be written as \eqref{eq:segint}, whereas
within smooth quartics with a line of the second kind of ramification type $R=2,1^2$ (resp. $R=2^2$),
those of the shape \eqref{eq:segint} have codimension 1 (resp. codimension 2).
\end{theo}
 

As a result, we verify that Segre's claims generally do not hold
for smooth quartics with a line of the second kind of ramification type $R=2,1^2$ or $R=2^2$
(see Sect.~\ref{s:corr2},~\ref{s:corr}). On the other hand, we show that if a quartic $\XXf$ is given by \eqref{eq:segint}, then both $\Ruledfourf$ and 
$\TLL_1, \ldots, \TLL_4$ are uniquely determined by $(\XXf, \ell)$ and can be effectively computed (see Prop.~\ref{prop-type1111}  and Cor.~\ref{cor:211segre},~\ref{cor:22segre}). 
In this case, the degree of the singular locus $\sing(\Ruledfourf)$ and the ramification type $R$ 
determine the Kodaira types of all singular fibers of \eqref{eq:fibration}  (compare Claim~\ref{eq-numberoflines}). 

\medskip

Our interest in 
surfaces with lines of the second kind has various reasons. 
Firstly, lines on surfaces play a key role in arithmetic considerations (see e.g. \cite[Thm.~2]{Voloch}, \cite{harris-caporaso}, \cite{harris-tschinkel}).
In the case of quartics, lines of the second kind appear (implicitly) for example 
in \cite[proof of Thm~5.1]{harris-tschinkel}.
Secondly,  even for $\KK = \CC$ the maximal number of lines on smooth degree-$d$ surfaces in $\PP^3_\KK$ remains unknown for $d \geq 5$. 
The existence of lines of the second kind is  one of the obstacles to improving the general bound
given in  \cite[$\S$~4]{Segre} (see e.g. \cite[Prop.~6.2]{boissieresarti}). 
A better understanding of the degree-$4$
case sheds  light on the question what happens for higher degrees,
which we plan to address in future work. 
Finally, one of our main aims is to clarify certain misconceptions/errors that appear in the classical paper \cite{Segre}, and thus correct the current picture of certain aspects
of the geometry of  quartics in $\PP^3$ that contain
a line.

\medskip

{\em Conventions:} Unless otherwise indicated, in this note we work over an algebraically closed field $\KK$ of characteristic $p \neq 2,3$. 
By abuse of notation, whenever it leads to no ambiguity,  we use the same symbol to denote a homogeneous polynomial and the set of its zeroes.

\section{Segre's construction of quartics with lines of the second kind} \label{sect-segreconstruction}

Let $\XXf \subset \PP^{3}_{\KK}$ be a {\bf smooth} quartic surface that contains a line $\ell$. For
 a point $P \in \ell$ we put $F_P$  to denote   the fiber of the fibration \eqref{eq:fibration} that contains $P$;
equivalently $F_P$ is the planar cubic 
residual to the line $\ell$ in the intersection of $\XXf$ with the  tangent space $\operatorname{T}_{P} \XXf$:
\[
X\cap T_PX = \ell + F_P.
\]

The restriction of the fibration \eqref{eq:fibration} to the line $\ell$ defines 
a degree $3$ map 
\begin{equation} \label{eq-map}
\pi|_\ell: \ell \rightarrow \PP^1 .
\end{equation}
Following \cite[$\S$~3]{ramsschuett} we put $\mathcal R$ to denote the ramification divisor of \eqref{eq-map}.
By the Hurwitz formula, one has $\deg(\mathcal R)=4$, so we distinguish 
 three possible {\bf ramification types}:   
$$ R = 1^4\;\; \text{ (i.e. $\mathcal R$ is a sum of four distinct points)}, \;\;\;R=2,1^2 \;\;
\text{ or } \;\; R=2^2. $$

In his attempt to describe all complex quartics with lines of the second kind, 
Segre uses the geometry of ruled surfaces in the classical sense, i.e surfaces covered by lines. 
For the convenience of the reader, in the remark below we collect  certain facts about ruled surfaces (in the classical sense), that can be found in \cite{puttop}.

\begin{rem} \label{japtop} Let $S \subset \PP^{3}_{\KK}$ be a reduced, irreducible surface that is a union of a family of lines
and which is not a cone over a planar curve (such surfaces were classically called  ruled surfaces, see \cite[p.~152]{puttop}).
The lines on $S$ give an algebraic subset $\tilde{C}$ of the Grassmanian $\mbox{Gr}(2,4)$.  
By \cite[Lemma~1.2]{puttop}
and \cite[Corollary~1.6]{puttop} the set $\tilde{C}$ consists of an irreducible curve $C$ and at most two points (so-called isolated lines - see \cite[Def.~1.4]{puttop}).
A line $\tilde{\ell}\subset S$ is a directrix iff it meets all lines from $C$ (equivalently, the curve $C$ is contained in the tangent space  $\mbox{T}_{\tilde{\ell}}\mbox{Gr}(2,4)$  of the Grassmanian
 at the point $\tilde{\ell}$).
\end{rem}

The following simple observation forms the backbone of Segre's considerations in \cite[$\S$~6]{Segre}.

\begin{constr} \label{obs-segretrick}
Let $\XXf \subset \PP^{3}_{\KK}$ be a smooth quartic surface that contains a line $\ell$. Assume that $\XXf$ is given by the equation
\begin{equation} \label{eq-segretrick}
\Ruledfourf + \TLL_1 \cdot \ldots \cdot \TLL_4 ,
\end{equation}
where the quartic $\Ruledfourf$ is a union  of a family of lines, each of which meets $\ell$,
and the planes $\TLL_1, \ldots, \TLL_4$ meet along the line $\ell$.
Then $\ell$ is a line of the second kind on $\XXf$.  
\end{constr}

\begin{proof}
Let $\ell' \subset S$, $\ell' \neq \ell$ be a line that 
is not contained in the union of the planes $\TLL_1, \ldots, \TLL_4$,
 and let $P$ be the point in $\ell \cap \ell'$. 
 By \eqref{eq-segretrick} we have
$$
\ell' \cdot (\ell + F_P) = \ell'\cdot X  = 4 P,
$$
so $P$ is an inflection point of the cubic $F_P$. Consequently, the line $\ell$ meets the flex locus of the  smooth fibers of the fibration
\eqref{eq:fibration} in infinitely many points, so $\ell$ is contained in its closure.
\end{proof}


On the other hand 
a line $\ell$ (not necessarily of the second kind)  on a smooth quartic $\XXf \subset \PP^{3}_{\KK}$
gives rise to a hypersurface that is covered by lines as follows.
Consider the set
\begin{eqnarray} \label{eq:familySl}
\;\;\;\;\;\;\;\;\;
{\mathfrak S}_{\ell} := \{      
\tilde{\ell}\subset\PP^3 \mbox{ a line}; \exists \, \, P \in \ell \mbox{ s.t. } \tilde{\ell} \subset \operatorname{T}_P \XXf
\mbox{ and } \ii(\tilde{\ell}, F_{P};P) \geq 2 \, \}, 
\end{eqnarray}
where
$\ii(\cdot)$ stands for the proper 
intersection multiplicity of the line $\tilde{\ell}$ with the planar cubic  $F_{P}$  in the point $P$
(inside the plane $\operatorname{T}_{P} \XXf$). 
One has the following observation: 

\begin{lemm} \label{lemma-ruledoctic}
The union  $\cup \{ \tilde{\ell} \, : \, \tilde{\ell} \in  {\mathfrak S}_{\ell} \} =: \Ruledeight$
of the family of lines $\mathfrak S_\ell$
is a hypersurface in $\PP^3$ of degree at most eight. 
\end{lemm}

\begin{proof} 
Without loss of generality we can assume that $\ell = \mbox{V}(x_3,x_4)$.
Let $f$ be a generator of the ideal ${\mathcal I}(\XXf)$.
We have
$$ f = \sum_{0<i+j\leq 4} \alpha_{i,j}  x_3^i x_4^j  \quad \mbox{with} \quad    \alpha_{i,j} \in \KK[x_1,x_2] \, \text{ homogeneous of degree $4-i-j$}. $$
Since $X$ is smooth, $\alpha_{1,0}$, $\alpha_{0,1}$  have no common roots;
 after a coordinate transformation 
 we can also assume that $\alpha_{1,0}$, $\alpha_{0,1}$  have no multiple roots. In particular, by \eqref{eq:familySl} 
the planes $\mbox{V}(x_3)$, $\mbox{V}(x_4)$ are not contained in $S_8$.

Let $P = (p_1: p_2: 0: 0)  \in \ell$.
Since 
$\operatorname{T}_P \XXf = \mbox{V}(\alpha_{1,0}(P) x_3 + \alpha_{0,1}(P) x_4)$, if $\alpha_{0,1}(p_1,p_2) \neq 0$, then the cubic 
 $F_{P}$ is given by the vanishing of the polynomial
$$
g(x_1,x_2,x_3) := f( \alpha_{0,1}(p_1,p_2)x_1, \alpha_{0,1}(p_1,p_2) x_2, \alpha_{0,1}(p_1,p_2) x_3, - \alpha_{1,0}(p_1,p_2) x_3)/x_3 .
$$
By direct computation, there exist $h_1$, $h_2 \in \KK[x_1, x_2]$ (resp. $h_3$) of degree $5$  (resp.~$8$) such that  
\begin{enumerate}
\item
$\frac{\partial g}{\partial x_3 }(p_1,p_2,0) =  \alpha^2_{0,1}(p_1,p_2) \cdot h_3(p_1,p_2)$  and 
\item $\frac{\partial g}{\partial x_j }(p_1,p_2,0) =  \alpha^3_{0,1}(p_1,p_2) \cdot h_j(p_1,p_2)
\mbox{ for } j=1,2.$
\end{enumerate}
We  define bihomogenous polynomials $H_8 \in \KK[z_1,z_2][x_1, \ldots,x_4]$ (resp. $H_3$)
of bidegree $(8,1)$ (resp. $(3,1)$) 
\begin{eqnarray*}
 H_8 &:=&  \alpha_{0,1}(z_1,z_2) \cdot h_1(z_1,z_2) \cdot x_1 + 
\alpha_{0,1}(z_1, z_2) \cdot h_2(z_1, z_2) \cdot x_2 + h_3(z_1, z_2) \cdot x_3 ,  \\
H_3 &:=& \alpha_{1,0}(z_1, z_2) \cdot x_3 + \alpha_{0,1}(z_1, z_2) \cdot x_4 \,  ,
\end{eqnarray*}
and put 
\[
{\mathfrak H}_8 = V(H_8)\subset\PP^1 \times \PP^3,\;\;\;  {\mathfrak H}_3=V(H_3)\subset\PP^1 \times \PP^3.
\]
By construction,  for a point 
 $P  \in \ell$ 
such that $\alpha_{0,1}(P) \neq 0$, 
 the union of the lines $\tilde{\ell} \in  {\mathfrak S}_{\ell}$ 
 through the point $P$ is given by
the intersection 
\[
\PP^3\supset
\mbox{V}(H_8(p_1,p_2)) \cap \mbox{V}(H_3(p_1,p_2)) = 
\cup\{\tilde{\ell} \in  {\mathfrak S}_{\ell}; P\in\tilde\ell\}.
\] 
Thus the union $S_8=\cup \{ \tilde{\ell} \, : \, \tilde{\ell} \in  {\mathfrak S}_{\ell} \}$
is contained in support of the push-forward $\mathfrak H$ of ${\mathfrak H}_8 \cap  {\mathfrak H}_3$ to $\PP^3$
via the projection map.
Clearly (up to an appropriate change of coordinates) $\mathfrak H$ is given by the resultant of $H_3$ and $H_8$ 
with respect to $z_1$ after dehomogenising $z_2=1$.
By inspection of the Sylvester matrix, we infer that $\mathfrak H\subset\PP^3$ has degree 11.

%

We continue by eliminating a superfluous component of $\mathfrak H$ 
that does not contribute to $S_8$.
To this end,
we assume that $\alpha_{0,1}(p_1,p_2) = 0$.
Then by inspection of the equations, 
$\{ (p_1: p_2: 0: 0) \} \times \mbox{V}(x_3)$ is a component of 
${\mathfrak H}_8 \cap {\mathfrak H}_3$. 
In consequence the plane
$\mbox{V}(x_3)$ occurs in $\mathfrak H$ with multiplicity $m \geq 3$ 
(once for each of the three distinct roots of $\alpha_{0,1}$).
Hence 
the degree of the effective divisor $(\mathfrak H - m\mbox{V}(x_3))$ 
in $\PP^3$ 
is at most $8$. By definition, we have
\begin{equation} \label{eq:incruled}
\cup \{ \tilde{\ell} \, : \, \tilde{\ell} \in  {\mathfrak S}_{\ell} , \,  \tilde{\ell} \nsubseteq  \mbox{V}(x_3) \} \subset 
\mbox{supp}(\mathfrak H -  m \mbox{V}(x_3)) \; \mbox{ and } \; \mbox{supp}(\mathfrak H) \setminus \mbox{V}(x_3) \subset S_8 .
\end{equation}
It remains to show the equality
\begin{equation} \label{eq-s8support}
S_8 = \mbox{supp}({\mathfrak H} -  m \mbox{V}(x_3)) .
\end{equation}
For this purpose
we use the fact that
the fiber $F_{P}$ is also given in $T_P X$ by the alternative equation
$$
f( \alpha_{1,0}(p_1,p_2)x_1, \alpha_{1,0}(p_1,p_2) x_2, -\alpha_{0,1}(p_1,p_2) x_4,  \alpha_{1,0}(p_1,p_2) x_4)/x_4 = 0,
$$
as soon as $\alpha_{1,0}(p_1,p_2) \neq 0$.
Repeating the above reasoning,
we derive an effective divisor $(\mathfrak H'- m' \mbox{V}(x_4))$ in $\PP^3$ of degree at most $8$ such that
 \begin{equation} \label{eq:incruled2}
\cup \{ \tilde{\ell} \, : \, \tilde{\ell} \in  {\mathfrak S}_{\ell} , \,  \tilde{\ell} \nsubseteq  \mbox{V}(x_4) \} \subset 
\mbox{supp}(\mathfrak H' -  m' \mbox{V}(x_4))
\end{equation}
The inclusions \eqref{eq:incruled},  \eqref{eq:incruled2} imply that $(S_8 \setminus \ell)$ is a hypersurface  in $\PP^3 \setminus \ell$.
Since $\ell$ meets all lines in  the family ${\mathfrak S}_{\ell}$, it lies in the closure of  $(S_8 \setminus \ell)$ and $S_8$ is 
a hypersurface. Finally, \eqref{eq-s8support} follows from \eqref{eq:incruled} 
(since $\mbox{V}(x_3)$ is not contained in $S_8$ by assumption).
\end{proof}

\begin{rem} 
Segre claims that $\Ruledeight$ is always of degree $8$ (see \cite[p.~92]{Segre}) and justifies that claim by a short, dynamic argument dealing with a line of 
the ramification type $R=1^4$.  
We were unable to track any proof of Lemma~\ref{lemma-ruledoctic} in the literature.
We will show that the above claim is true for 
a general quartic with a line of the second kind (Propositions \ref{prop:geninmT}, \ref{prop-generic-in-mZ}), 
and for all quartics with  the above ramification type $R=1^4$ (Lemma \ref{lemma-splitplane} (d)).
However, if we treat  $\Ruledeight$ as a \emph{ reduced } surface,
then the above claim does not always hold (see Cor.~\ref{cor:211segre}~(b), Cor.~\ref{cor:22segre}~(b)).
\end{rem}

\section{Segre's claims about complex quartics with lines of the second kind} \label{sect-segreclaims}

In \cite[$\S$~6]{Segre} Segre claims that each complex quartic  with a line of the second kind
can be constructed as in Construction~\ref{obs-segretrick}.
One of the key points in his arguments 
deals with components of $S_8$ in case $\KK=\CC$
(see the erroneous Claim \ref{eq-splittingquartic}).
Here we recall Segre's claims before correcting them in the following sections.

Crucially one can find components of $\Ruledeight$ by the following lemma.
Here we give a proof over an arbitrary algebraically closed field $\KK$
of characteristics $\neq 2,3$
based on the use of elliptic fibrations in \cite{ramsschuett}; Segre's original argument over $\CC$
uses convergent series (see \cite[p.~91]{Segre}).

\begin{lemm} \label{lemma-splitplanesegre}
Let $X\subset\PP^3_{\KK}$ be a smooth quartic and $\ell\subset X$ a line of the second kind.
If  $P \in \operatorname{supp}(R_{\ell})$,
then the tangent plane $\operatorname{T}_{P}\XXf$ is a component of the ruled surface $\Ruledeight$.
\end{lemm}

\begin{proof}
We claim that 
$P$ is a singular point of the fiber $F_P$.
Given this, any line in $\operatorname{T}_{P}\XXf$ through $P$ meets $F_{P}$ with multiplicity at least 2 in $P$, 
and the inclusion $\operatorname{T}_{P}\XXf \subset \Ruledeight$ 
results directly from \eqref{eq:familySl}.  

For $P$ of multiplicity $2$ in $\mathcal R$,
the claim follows from \cite[Lemma~3.2]{ramsschuett} and its proof therein.
On the other hand, if $P$ has multiplicity $1$ in $R_\ell$,
then $F_{P}$ is a cuspidal cubic  (Kodaira type $II$) by \cite[Lemma~3.2]{ramsschuett}.
By assumption $\ell$ meets $F_{P}$ in two distinct points.
Since $\ell$  lies in the closure of the flex locus of smooth fibers  of \eqref{eq:fibration},
the intersection points can only be the unique smooth inflection point
of the fiber and the cusp; in particular, $P$ equals the cusp. 
\end{proof}

%


Lemma~\ref{lemma-splitplanesegre} implies that the surface $\Ruledeight$ contains 
four (resp. three or two) planes when $\ell$ is a line of the second kind with ramification type $R=1^4$ (resp. $R=2,1^2$ or $2^2$). 
According to Segre, however, the hypersurface $\Ruledeight$ contains {\it always} four planes over $\CC$:


\begin{claim}[{\cite[$\S$~6., p.~92]{Segre}}] \label{eq-splittingquartic}
If $\ell$ is a line of the second kind, then  the octic $\Ruledeight$ is always a union 
$$\Ruledfour \cup \Pi_1 \cup \ldots \cup \Pi_4,$$ where $\Ruledfour$ is a ruled quartic,
the line $\ell$ is its directrix and the planes 
$\Pi_1$, $\ldots$, $\Pi_4$ contain $\ell$.
\end{claim}

This leads to another claim:

\begin{claim}[{\cite[$\S$~6., p.~94]{Segre}}] \label{eq-generalequation}
Each quartic with a line of the second kind is given by 
\eqref{eq-segretrick}. Moreover, the singular locus of $\Ruledfourf$ is either a twisted cubic of double points or a line of triple points.  
\end{claim}
The latter is used by Segre to show the following statement. 

\begin{claim}[{\cite[$\S$~7., p.~95]{Segre}}] \label{eq-numberoflines}
 A line of the second kind on a smooth quartic  is met by $12$, $15$ or $18$ other lines on $\XXf$, with the exact number of lines meeting $\ell$ given by the formula  
$$ (24 - 3 \cdot \#(\operatorname{supp}(\mathcal R)). $$
\end{claim}

As a consequence Segre derives the following bound. 
\begin{claim}[{\cite[$\S$~7]{Segre}}] \label{eq-minibound}
 A line on a smooth quartic  is met by at most  $18$ other lines on the surface in question.
\end{claim}

In Sect.~\ref{s:1^4} (in combination with \cite{ramsschuett})
we will verify that Segre's claims are correct for quartics with  lines of the second type of the  ramification type $1^4$ in any characteristic $\neq 2,3$.
In Sections \ref{s:corr2},~\ref{s:corr}, however,  we will give counterexamples to the first three claims for the other two ramification types, i.e.
$R=2^2$ and $2,1^2$. We will also discuss some aspects of the  geometry of all quartics that violate the bound  of Claim~\ref{eq-minibound},
but first we will need some technical preliminaries.

\section{Technical preliminaries} \label{sect-tech-prel}

In our considerations  we will need  the following 
lemma about ruled quartics. 

\begin{lemm} \label{lemma-linewayfromsingularities}
Let $\Ruledfourf \subset \PP^{3}_{\KK}$ be a hypersurface of degree four that is 
a union of a family ${\mathfrak F}$ of lines,
 $\Ruledfourf = \cup \{ \tilde{\ell} \, : \, \tilde{\ell} \in  {\mathfrak F} \}$.
If there exists a line $\ell \subset \Ruledfourf$ such that
$$
\ell \cap \sing(\Ruledfourf) = \emptyset \mbox{ and  } \ell \cap \tilde{\ell} \neq \emptyset \mbox{ for all } \tilde\ell \in {\mathfrak F}  ,
$$
then one has
\begin{itemize}
\item
$\Ruledfourf$ is not a cone over a quartic curve, 
\item
the line $\ell$ is its directrix and 
\item
the singular locus 
$\sing(\Ruledfourf)$ consists either of a line of triple points, or of a twisted cubic of double points.
\end{itemize}
\end{lemm}

\begin{proof}
The line $\ell$ is  contained in the smooth locus of $\Ruledfourf$ and meets all its components
(since each component by assumption is covered by lines from $\mathfrak F$), 
so the quartic in question  
is irreducible. We claim that $\Ruledfourf$  is not a cone. Indeed,
if the hypersurface were  a cone, then $\ell$ would not run through its vertex and 
hence $\Ruledfourf$ would contain a plane.  
Thus $\Ruledfourf$ is a ruled quartic in  the sense of \cite{puttop}. 
The line $\ell$ is a directrix of $\Ruledfourf$
because it meets infinitely many  lines in $\mathfrak F$.
As for the singular locus, one can continue by ruling out all other possibilities from \cite[$\S$~3.2]{puttop}.
For space reasons, we omit the details here.
\end{proof}

As an immediate consequence we obtain the following corollary.

\begin{cor} \label{cor:singloc}
If a smooth quartic $\XXf\subset\PP^3_\KK$ is given by Construction~\ref{obs-segretrick},
then $\Ruledfourf$ is a ruled quartic (in the classical sense) and 
the singular locus 
$\sing(\Ruledfourf)$ consists either of a line of triple points, or of a twisted cubic of double points.
\end{cor}


If a smooth quartic $\XXf$ is  given by Construction~\ref{obs-segretrick},
then the intersection  
$\TLL_j \cap \Ruledfourf$ consists of four lines. 
On the fibration \eqref{eq:fibration},
this gives a singular fiber of Kodaira type $I_3$ or $IV$
comprising the 3 lines other than $\ell$.
Without difficulty one can  distinguish these two cases according to the shape of sing$(S)$:

\begin{lemm} \label{cor:singfib}
Let  $\XXf$ be a smooth quartic  given by Construction~\ref{obs-segretrick}.
\begin{enumerate} 
\item[(a)] 
The four planes $\TLL_1, \ldots, \TLL_4$ contain  fibers of Kodaira type $IV$ of the 
fibration \eqref{eq:fibration} 
iff
$\sing(\Ruledfourf)$ is a line of triple points. 
\item[(b)] 
The 
four planes $\TLL_1, \ldots, \TLL_4$ contain  fibers of Kodaira type $I_3$ of the 
fibration \eqref{eq:fibration} 
iff
$\sing(\Ruledfourf)$ is a twisted cubic. 
\end{enumerate}
\end{lemm}

We can now derive quite easily the following uniqueness result:

\begin{lemm} \label{lem:4eqplanes}
Let $\XXf$ be  a smooth quartic obtained by  Construction~\ref{obs-segretrick}. 
 If $\XXf$ is given by the polynomial 
$\tilde{\Ruledfourf} + \TLL_1 \cdot \ldots \cdot  \TLL_4$, 
where $\tilde{\Ruledfourf}$ is a  ruled quartic and 
the line $\ell $ is its directrix, then  $$\Ruledfourf = \tilde{\Ruledfourf}.$$
\end{lemm}

\begin{proof}
The quartic $\XXf$ is smooth, so $\Ruledfourf$ has no  singularities on the line $\ell$. 
We can apply Lemma~\ref{lemma-linewayfromsingularities}, to see that 
 $\sing(\Ruledfourf)$  is either a line of triple points or a twisted cubic of double points, 
and likewise for $\sing(\tilde{\Ruledfourf})$. 
For $j= 1, \ldots, 4$, we have the following intersections consisting of 4 lines:
\begin{equation} \label{eq:tllhsect0}
\TLL_j \cap \tilde{\Ruledfourf} =  \TLL_j \cap \XXf \stackrel{(\ast)}{=} 
 \TLL_j \cap \Ruledfourf.
\end{equation}
Suppose that  $\sing(\Ruledfourf)$ is a line of triple points.  Let $j = 1, \ldots, 4$.
Then 
$\TLL_j$  contains a fiber of Kodaira type $IV$ of the fibration \eqref{eq:fibration}
by Lemma \ref{cor:singfib}.
Moreover, \eqref{eq:tllhsect0}  yields that $\sing(\tilde{\Ruledfourf})$ meets the plane $\TLL_j$ only in the  singular  point
of the type $IV$ fiber.
In particular,  $\sing(\tilde{\Ruledfourf})$  cannot be a twisted cubic, so it is a line as well.  
It  meets  $\sing(\Ruledfourf)$ at singular points of the four type $IV$ fibers, so 
the lines of singularities coincide, i.e.  
\[
\sing(\Ruledfourf) = \sing(\tilde{\Ruledfourf}).
\]
Thus the irreducible quartics $\Ruledfourf$, $\tilde{\Ruledfourf}$ meet along the line $\sing(\Ruledfourf)$ with multiplicity at least $9$. Furthermore, they contain the $13$ lines in 
\eqref{eq:tllhsect0}, so B\'ezout implies 
$\Ruledfourf = \tilde{\Ruledfourf}$.

On the other hand,
if  $\sing(\Ruledfourf)$ is a twisted cubic of double points,
then $\TLL_j$ contains a fiber of Kodaira type $I_3$ of the fibration \eqref{eq:fibration}
for each $j=1, \ldots, 4$.
By \eqref{eq:tllhsect0},  $\sing(\tilde{\Ruledfourf})$  is also a twisted cubic;
it meets $\sing(\Ruledfourf)$ in $12$ points, so they coincide for degree reasons. 
Hence the irreducible quartics $\Ruledfourf$, $\tilde{\Ruledfourf}$ meet along this twisted 
cubic  with multiplicity at least $4$. 
Since they 
 have $13$ lines in common, we obtain $\Ruledfourf = \tilde{\Ruledfourf}$ by B\'ezout again.
\end{proof}

The assumption that $\XXf$ is given by  Construction~\ref{obs-segretrick} has another consequence for the singular fibers of 
the 
fibration \eqref{eq:fibration}. 
\begin{lemm}
\label{lem:ramIV}
Let  $\XXf$ be a smooth quartic  given by Construction~\ref{obs-segretrick} and let $F$ be a fiber of the 
fibration \eqref{eq:fibration}. If $F$ meets the line $\ell$ in exactly one point $P$, then $F$ is of Kodaira type $IV$.
\end{lemm}
\begin{proof}
We can assume that $F \subset \mbox{V}(x_4)$, the line $\ell = \mbox{V}(x_3, x_4)$ and $P=(0:1:0:0)$. Since $\ell$ is a directrix of $\Ruledfourf$, the intersection
$\XXf \cap  \mbox{V}(x_4)$  is given by product of four linear forms:
$$
x_3 \cdot \prod_{i=1}^3 (a_i x_1 + b_i x_2 + c_i x_3) .
$$
By \eqref{eq-segretrick}, the fiber $F$ is given as 
$$
\prod_{i=1}^3 (a_i x_1 + b_i x_2 + c_i x_3) + \gamma x_3^3  \quad \mbox{ for some } \gamma \in \KK.
$$
But $P \in F$ implies $b_1 = b_2 = b_3 = 0$,
 and we obtain the 3 fiber components meeting at $(0:1:0)$ as required 
 (recall that $\XXf$ is smooth, so the intersection $\XXf \cap  \mbox{V}(x_4)$ is reduced).
\end{proof}

In the sequel we will also use some properties of the surface $\Ruledeight$. 
\begin{lemm} \label{lem:planeS8}
Let $\ell$ be a line of the second kind on a smooth quartic $\XXf$ and
let  $\Ruledeight$ be the surface defined by the pair $(\XXf, \ell)$.
\begin{enumerate}

\item[(a)]
If a plane $\Pi$ is a component of the surface $\Ruledeight$, then  $\Pi = \mbox{T}_{P}\XXf$ for a point $P \in \operatorname{supp}(\mathcal R)$.

\item[(b)]
The surface $\Ruledeight$ does not decompose into planes neither does it contain an irreducible quadric. 
\end{enumerate}
\end{lemm}
\begin{proof} {\sl (a)} Let the plane $\Pi$ be a component of $\Ruledeight$.  
Each line $\ell' \in {\mathfrak S}_{\ell}$ is contained in a component of $\Ruledeight$,
 so infinitely many are contained in $\Pi$. If two of these lines run 
through the same point on $\ell$, the point is a singularity of the fiber and we are done.
Otherwise,  $\Pi$ meets $\ell$ in two points, so it contains $\ell$.
 Let $F \subset \Pi$ be the  fiber of \eqref{eq:fibration}.
Suppose that $F$ meets $\ell$ in three distinct points.  Then, by definition of the family ${\mathfrak S}_{\ell}$ (see \eqref{eq:familySl}),  
exactly three lines contained in $\Pi$ are members of
${\mathfrak S}_{\ell}$, so $\Pi \cap \Ruledeight$ consists of three lines and $\Pi$ is no component of $\Ruledeight$.

{\sl (b)} 
We consider the planes contained in $\Ruledeight$ 
and denote by $S_k$ the residual hypersurface of degree $k$.
We shall use the following elementary fact:
A general hyperplane $H$ containing $\ell$ meets $S_8$ and $S_k$ in the same 4 lines,
namely $\ell$ and the 3 flex tangents 
of the cubic residual to $\ell$ in $H$
at the intersection points with $\ell$.
This immediately implies that $k>0$,
that is, $S_8$ does not decompose into planes.

Suppose that $\Ruledtwo \subset S_k$ is an irreducible quadric. Observe that  $\Ruledtwo$ contains the line $\ell$.
If $S_2$ is a cone, then its vertex $P$ lies on $\ell$.
By \eqref{eq:familySl} the quadric $\Ruledtwo$ contains the tangent plane $\operatorname{T}_{P}\XXf$, contradiction. 
Thus $\Ruledtwo$ is smooth. 
Recall the intersection of $S_k$ with a general hyperplane $H$ containing $\ell$:
\[
S_k\cap H = \ell+\ell_1+\ell_2+\ell_3.
\]
By the quadric structure we have, say, $\ell_1\subset S_2$ and $\ell_2,\ell_3\not\subset S_2$.
But then $S_2$ clearly contains a unique line $\tilde\ell\neq\ell,\ell_2$ 
through the intersection point $P=\ell\cap\ell_2$.
That is, $S_k$ contains two lines through $P$  other than $\ell$.
However, through a general point $P\in\ell$,
$S_k$ 
 contains exactly one line other than $\ell$, 
namely the flex tangent $T_P F_P$
(inside $T_P X$).
This gives the required contradiction. 
\end{proof}

We end this section with a lemma that follows one of the main ideas of \cite[$\S$~6]{Segre}.  

\begin{lemm} \label{lem:S4tr}
Let $\ell$ be a line of the second kind on a smooth quartic $\XXf$ and
let  $\Ruledeight$ be defined by the pair $(\XXf, \ell)$. Moreover, 
let $\operatorname{supp}(\mathcal R) = \{P_1, \ldots, P_k \} \; (k=2,3,4)$.
Assume that
\begin{itemize}
\item there exists a quartic $\Ruledfour$ that contains none of the planes $\mbox{T}_{P_i}\XXf$, where $i = 1, \ldots, k$, and the following equality holds 
\begin{equation} \label{eq:decS4tr}
\Ruledeight = \Ruledfour \cup \mbox{T}_{P_1}\XXf \cup \ldots \cup \mbox{T}_{P_k}\XXf \, 
\end{equation}
\item  there exist pairwise distinct  planes $L_1$, $\ldots$, $L_4$, none of which coincides with 
$\mbox{T}_{P_1}\XXf$, $\ldots$, $\mbox{T}_{P_k}\XXf$,
 that contain $\ell$ and  intersect $\XXf$ along four lines, i.e.  for $j = 1, \ldots, 4$ one has:
\begin{equation} \label{eq:4lS4tr}
L_j \cap \XXf = \ell + \ell_j + \ell'_j + \ell''_j .
\end{equation}
\end{itemize}
Then $\XXf$ is given by the equation
\begin{equation} 
\label{eq:segre1111}
\XXf = \lambda  \cdot \Ruledfour +  L_1 \cdot \ldots \cdot L_4 \quad  \mbox{ for } \lambda \in \KK^\times.
\end{equation}
\end{lemm}
\begin{proof}
We consider the proper intersection cycle of the two quartics
$$
\Ruledfour \cdot \XXf  =  m \, \ell + \DivisorRest,
$$
where $\ell$ is not contained in the support of the residual divisor $\DivisorRest$ of degree $(16-m)$.
We claim that $m \geq 4$. Indeed, fix a point $P \in (\ell \setminus \operatorname{supp}(\DivisorRest))$ 
such that the fiber $F_P$ is a smooth;
by \cite[Lemma~3.2]{ramsschuett} (see also the proof of Lemma~\ref{lemma-splitplanesegre}), 
$F_P$ meets the line $\ell$ in three distinct points. Let $\tilde{\ell} := T_{P} F_{P}$ and let  
$\Pi \neq T_P\XXf$ be  a plane 
that contains $\tilde{\ell}$. Then
\begin{equation} \label{eq-m}
m = \ii( \Ruledfour \cap \Pi, \XXf \cap \Pi; P).
\end{equation} 
Since $\ell$ is a line of the second kind, $P$ is an inflection point of $F_P$, which yields the equality
$$
\tilde{\ell} \cdot  \XXf = 4 \, P.
$$
Thus $\tilde{\ell}$ meets the planar quartic  $\XXf \cap \Pi$ in the point $P$ with multiplicity $4$.
But $\tilde{\ell}$ is a component of the planar curve $\Ruledfour \cap \Pi$ so the 
right-hand side of \eqref{eq-m} is at least 4 as claimed.


By definition, the surface $\Ruledeight$ contains all lines on $\XXf$ that meet $\ell$. We infer from \eqref{eq:decS4tr}, that the quartic $\Ruledfour$, 
and consequently the support of the divisor $\DivisorRest$ contain
 four triplets $\ell_j, \ell'_j, \ell''_j$   of coplanar lines,
where $j=1, \ldots, 4$. For degree reasons we thus find $m=4$ and
$$ \Ruledfour \cdot \XXf  =  4 \, \ell + \sum_{j=1}^4  (\ell_j + \ell'_j + \ell''_j) . $$

Since 
$\Ruledfour$ and $L_1 \cdot \ldots \cdot L_4$  intersect  $\XXf$ along the same divisor, one obtains \eqref{eq:segre1111} 
which completes the proof of Lemma \ref{lem:S4tr}. 
\end{proof}

\section{Ramification type $1^4$} \label{sect-1111}
\label{s:1^4}

In this section we 
study quartics with a  line of the second kind of ramification type $R=1^4$.
Recall from Theorem \ref{thm} that it is exactly this ramification type
where Segre's claims  listed in Sect.~\ref{sect-segreclaims} ought to hold true
(cf.~\cite[\S 6]{Segre}). 
Our purpose is to give precise proofs 
of these claims 
over the field $\KK$ (not only $\CC$).

In the lemma below we collect certain consequences of \cite[$\S$~3]{ramsschuett} that we will need in the sequel.

\begin{lemm} \label{lemma-splitplane}
\label{lem:5.1}
Let $X\subset\PP^3$ be a smooth quartic
and $\ell\subset X$ be a line of the second kind 
with ramification type $1^4$ and $\mathcal R = P_1 + \ldots + P_4$.

\begin{enumerate}
\item[(a)] The fibers $F_{P_1}$, $\ldots$, $F_{P_4}$ are singular fibers  of type $II$. In particular, they contain no lines.

\item[(b)] 
The fibration \eqref{eq:fibration} has no singular fibers of type $I_2$.

\item[(c)]
 The line $\ell$ is met by exactly $12$ lines $\ell' \neq \ell$ on $\XXf$. They form four triplets of coplanar lines.

\item[(d)] 
The tangent planes $\operatorname{T}_{P_j}\XXf$, where $j = 1, \ldots, 4$,  are components of the surface $\Ruledeight$ and 
one has $\deg(\Ruledeight) = 8$.

\item[(e)] 
The quartic $\Ruledfour$ residual to the tangent spaces $\operatorname{T}_{P_1}\XXf$, $\ldots$,
 $\operatorname{T}_{P_4}\XXf$ in $\Ruledeight$ is a union of a family of lines, each of which meets $\ell$. 
 Moreover,  $\ell$ is no component of $\sing(\Ruledfour)$.
 \end{enumerate}
\end{lemm}

\begin{proof} {\sl (a)} By \cite[Lemma~3.2]{ramsschuett} the fibers $F_{P_1}, \ldots, F_{P_4}$ are cuspidal cubics (i.e. fibers of type $II$).

 {\sl (b)} The fibers of type $I_2$ are ruled out by \cite[Lemma~3.1]{ramsschuett} and \cite[Lemma~3.2]{ramsschuett}.

{\sl (c)} \cite[Prop.~4.1]{ramsschuett} implies that the line $\ell$ is met by exactly $12$ lines $\ell' \neq \ell$ on $\XXf$. 
By \cite[Lemma~3.1]{ramsschuett} they are contained in fibers of type either $I_3$ or $IV$ of the elliptic fibration \eqref{eq:fibration},
 so they form triplets of coplanar lines.

{\sl (d)} 
The first assertion follows from Lemma \ref{lemma-splitplanesegre}.
Let $\Ruledfour$ denote the hypersurface residual to the tangent spaces $\operatorname{T}_{P_1}\XXf$, $\ldots$,
 $\operatorname{T}_{P_4}\XXf$ in $\Ruledeight$.
A plane $\Pi \neq \operatorname{T}_{P_j}\XXf\;\;  (j=1, \ldots, 4)$ containing $\ell$ 
 meets  
$\Ruledfour$ along four distinct lines
($\ell$ and the flex tangents to the 3 intersection points of $\ell$ and the residual plane cubic).
Hence $\deg(\Ruledfour) \geq 4$.
Since $\deg(\Ruledeight)\leq 8$ by Lemma \ref{lemma-ruledoctic},
this completes the proof of (d). 


{\sl (e)} 
The quartic $\Ruledfour$ is covered by the planes containing $\ell$.
Each  intersection with a plane outside the ramified points $P_1,\hdots,P_4$
decomposes into 4 lines by the proof of {\sl (d)}.
On the other hand the intersection of the quartic $\Ruledfour$ and the plane $\operatorname{T}_{P_j}\XXf$, where $j=1, \ldots, 4$, contains at least three  distinct lines: $\ell$, 
the line  tangent to $F_{P_j}$ in the  unique smooth inflection point,
and the tangent cone $\operatorname{C}_{P_j}F_{P_j}$. 
For degree reasons, it thus consists solely of lines (each meeting $\ell$).
Hence,  $\Ruledfour$ is the union of the family of lines
given by intersecting with the planes containing $\ell$. 

By (c), $\ell$ is met by 12 lines on $X$ comprising four singular fibers of \eqref{eq:fibration}.
On these fibers, the flex tangents to the 3 intersection points of $\ell$ and the fiber
are by definition the lines themselves. Thus $\Ruledfour$ contains all lines on $X$ meeting $\ell$.

By the proof of {\sl (d)} again,  intersecting of $\Ruledfour$ with a general  plane containing $\ell$,
we obtain $\ell$ as a reduced component of the intersection for degree reasons;
hence $\ell$  cannot be contained in $\sing(\Ruledfour)$. This completes the proof of (e).
\end{proof}



In the proposition below, ruled surface 
stands for  a reduced, irreducible surface that is a union of a family of lines. Moreover, a ruled surface is assumed  not to be a cone
(see Remark~\ref{japtop}).  As in Sect.~\ref{sect-segreconstruction} we put  $\Ruledeight$ to denote  the octic defined by the pair $(\XXf, \ell)$ (see Lemma~\ref{lemma-ruledoctic}).

\begin{prop} \label{prop-type1111} 
\label{prop:5.2}
Let $\XXf\subset\PP^3$ be a smooth quartic
and $\ell \subset \XXf$ be a line of the second kind with ramification  type $R=1^4$. Then

{\sl (a)} The quartic $\XXf$ is given by the equation \eqref{eq:segre1111}, where
\begin{itemize}
\item the   planes $L_1, \ldots, L_4$ are pairwise distinct and meet along the line $\ell$,
\item  each plane $L_j$ intersects $\XXf$ along four lines,
\item the  ruled quartic $\Ruledfour$ is a component of the surface $\Ruledeight$.
\end{itemize}
Moreover, the line $\ell$ is a directrix of $\Ruledfour$,  
the quartic $\Ruledfour$ is smooth along $\ell$ and the singular locus
$\sing(\Ruledfour)$ consists either of a line of triple points, or of a twisted cubic of double points.

{\sl (b)} The decomposition \eqref{eq-segretrick} for the pair $(\XXf, \ell)$ is unique and coincides with \eqref{eq:segre1111}, i.e. $\Ruledfourf = \Ruledfour$
and $\{ \TLL_1, \ldots,\TLL_4 \} = \{L_1, \ldots,L_4\}$.

{\sl (c)} The fibration \eqref{eq:fibration} has singular fibers $4 \, I_3  \oplus   4 \, I_1 \oplus 4 \, II $,  iff
$\sing(\Ruledfour)$ is a twisted cubic.

{\sl (d)} The fibration \eqref{eq:fibration} has singular fibers $4 \, IV  \oplus 4 \, II$ iff
$\sing(\Ruledfour)$ is a line of triple points. 
\end{prop}
\begin{proof}
{\sl (a)} Lemma~\ref{lemma-splitplane}~(d),~(e) implies that the equality  \eqref{eq:decS4tr} holds. By   Lemma~\ref{lemma-splitplane}~(a)
the planes $ \mbox{T}_{P_1}\XXf, \ldots, \mbox{T}_{P_4}\XXf$ contain no lines $\ell' \neq \ell$ on $\XXf$. 
We infer from  Lemma~\ref{lemma-splitplane}~(c) that there exist four planes $L_1, \ldots, L_4$ that satisfy the assumptions of  Lemma~\ref{lem:S4tr}
(see \eqref{eq:4lS4tr}). Thus  $\XXf$ is given by \eqref{eq:segre1111}.

The quartic $\XXf$ is smooth, 
so $\Ruledfour$ has no singularities on the line $\ell$.
 By  Lemma~\ref{lemma-splitplane}~(e)
we can use  Lemma~\ref{lemma-linewayfromsingularities}
to complete the proof of (a).

{\sl (b)} 
Suppose that $\XXf$ is given by 
the equation \eqref{eq-segretrick},
where the ruled quartic $\Ruledfourf$ and the planes $\TLL_1,  \ldots, \TLL_4$ satisfy the conditions of Construction~\ref{obs-segretrick}.
Obviously the equality ($\ast$) of \eqref{eq:tllhsect0} holds, so the plane     $\TLL_j$   intersects the quartic $\XXf$ along four lines.  
Thus Lemma~\ref{lemma-splitplane}~(c) yields the equality
$$
 \{ \TLL_1,  \ldots, \TLL_4 \} = \{L_1,  \ldots, L_4\}.
$$
Claim~(b) follows directly from Lemma~\ref{lem:4eqplanes}.

{\sl (c), (d)} By Lemma~\ref{lemma-splitplane}~(a),
 the fibration~\eqref{eq:fibration} has four fibers of Kodaira type $II$.
In case {\sl (d)}, there are also 4 fibers of Kodaira type $IV$ by Lemma \ref{cor:singfib} (a). 
Thus Euler number considerations prove the claim.
Similarly, Lemma \ref{cor:singfib} (b) shows that there are 4  $I_3$ fibers in case {\sl (c)}.
 Since fibers of Kodaira type $I_2$ are ruled out by Lemma~\ref{lemma-splitplane}~(b),
there are 4 other singular fibers,
each of Kodaira type $I_1$, as follows from \cite[Lemma~4.2]{ramsschuett}.
%
%
%
%
\end{proof}



\section{Ramification type $2,1^2$}
\label{s:corr2}

In this section we determine a $6$-parameter family of quartics that (up to projective equivalence) contains
all smooth quartics with a line of the second kind and ramification type $R=2,1^2$.
This enables us to show that neither Claim~\ref{eq-splittingquartic}, \ref{eq-generalequation} nor
\ref{eq-numberoflines}
hold for this ramification type. 
We put
\begin{equation} \label{eq:rl211}
\mathcal R = 2 P_1 + P_2 + P_3 .
\end{equation}
At first we develop a suitable projective model for the quartics
of this ramification type, much like the family $\mZ$ determined for ramification type $R=2^2$ in \cite{ramsschuett} 
(see Sect.~\ref{s:corr}).

\begin{lemm}
\label{lem:2,1^2}
Let $\ell$ be a line of the second kind on a smooth quartic $\XXf$ with ramification type $R=2, 1^2$.
Then $\XXf$ is projectively equivalent to a quartic in the family $\mT$ given by the polynomials
\begin{eqnarray*}
\mT: & g
-4  c^3 (b+4ac^3) x_1 x_2 x_4^2+ b x_1 x_2x_3 x_4+ a x_1 x_2x_3^2+ x_2^2(x_2-3c x_1)x_4\\
& \;\;\;+ x_1^3 x_3+c(x_3-4c^3x_4)((b+4ac^3)x_4+ax_3)^2(cx_1+x_2)/3 , 
\end{eqnarray*}
where $a,b,c\in\KK, c\neq 0$, and $g \in \KK[x_3, x_4]$ is homogeneous of degree $4$.
\end{lemm}

\begin{proof}
By a linear transformation, we can assume 
that the line $\ell$ is given by $x_3=x_4=0$,
and that the ramification occurs doubly at $0$  and simply at $\infty$, that is, at $x_3=0$ and $x_4=0$.
A further normalisation makes the residual cubic polynomials in these fibers $x_1^3, x_2^2(x_2-3cx_1)$
for some $c\in\KK^\times$.
The equation thus becomes
\begin{eqnarray}
\label{eq:mid}
x_3x_1^3+x_4x_2^2(x_2-3cx_1)+x_3^2q_1+x_3x_4q_2+x_4^2q_3=0 , 
\end{eqnarray}
where  $q_1, q_2, q_3$ are homogeneous quadratic forms in $x_1,\hdots,x_4$.
Translations of $x_1$ and $x_2$ in linear terms in  $x_3, x_4$ ensure 
that the terms involving $x_1^2x_3$ and $x_2^2x_4$ in \eqref{eq:mid} are zero.
Then we solve for $\ell$ to be a line of the second kind,
i.e.~for the Hessian of \eqref{eq:mid} to vanish identically on $\ell$.
The corresponding system of equations can be solved directly,
resulting in  the given equation for $\mT$.
\end{proof}

\begin{rem}
\label{rem:deg}
For the ramification type of a smooth quartic $X\in\mT$
to really be $R=2,1^2$,
we need $c\neq 0$,
for otherwise the two fibers of Kodaira type $II$ come together to a single fiber of type $IV$
meeting the line in 
the singularity of the fiber, so the ramification type changes to $R=2^2$.
\end{rem}

Recall that we defined $P_1$ to be the only non-reduced component of $\mathcal R$.
For all $\XXf \in \mT$, the fiber $F_{P_1}$ is   contained in the plane  $\mbox{V}(x_3)$. By definition of
the  ramification type,
it is the only fiber of \eqref{eq:fibration} that meets the line $\ell$ in exactly one point.

\begin{lemm} \label{lem:ftype112}
\label{lem:6.3}
(a) A general quartic $\XXf \in \mT$ is smooth and  the fibration \eqref{eq:fibration} is of type $5 I_3 \oplus 5 I_1 \oplus 2 II.$
In particular, the  fiber $F_{P_1}$ is of Kodaira type $I_1$,
and $\ell$ meets exactly 15 lines on $X$.

(b)  Let  $\XXf \in \mT$ be smooth. Then

\begin{itemize}

\item the fiber $F_{P_1}$ 
 is of type $I_2$ iff the coefficient of $x_3^4$ in $g$  equals $4a^3c^3/27\neq 0$,

\item the fiber $F_{P_1}$ 
is of type $IV$ if and only if $a = 0$,

\item the fibers $F_{P_2}$, $F_{P_3}$ are always of type $II$ if the ramification type is $R=2,1^2$.
\end{itemize}
\end{lemm}

\begin{proof} {\sl (a)}
Singularities of $X$ cause the elliptic fibration \eqref{eq:fibration} induced by the line $\ell$ 
to either attain (more) reducible fibers or degenerate completely.
Presently the singular fibers at the ramification points $P_1,\hdots,P_3$ 
have Kodaira types $I_1$ at $x_3=0$ and $II$ at $x_4=0$
and $x_3=4c^3x_4$.
Generically, the discriminant reveals 
that there are 5 reducible fibers, each of Kodaira type $I_3$.
One  checks that each fiber corresponds to a hyperplane in $\PP^3$ whose intersection with the quartic
splits into $\ell$ and three residual lines.
Hence a general quartic $X\in\mT$ is smooth.
Since the lines met by $\ell$ appear as components of singular fibers of \eqref{eq:fibration},
$\ell$ generally meets exactly 15 lines on $X$.

{\sl (b)} 
The residual cubic $F_{P_1}$ generally has degree one in $x_2$; it is given by the polynomial
\[
ax_3(acx_3+3x_1)x_2+3\, \mbox{coeff}(g,x_3^4) \,x_3^3+c^2a^2x_1x_3^2+3x_1^3.
\]
Therefore the cubic becomes reducible if and only if
\begin{itemize}
\item
either the linear term in $x_2$ vanishes identically, i.e.~$a=0$ and the  fiber is of the  type $IV$,
\item
or the factor $(acx_3+3x_1)$ of the coefficient of the linear term in $x_2$ divides the constant term as well,
i.e.~$\mbox{coeff}(g,x_3^4)=4a^3c^3/27$ and the fiber type is $I_2$ for $a\neq 0$.
\end{itemize}
%
%
The last claim of (b) follows from  \cite[Lemma~3.2]{ramsschuett}.
\end{proof}

As an immediate corollary of Lemma~\ref{lem:ftype112}~(b), we obtain the criterion when $\ell$ is met by $16$ lines on $\XXf \in \mT$.
\begin{cor}
\label{prop:15,16}
Let $X\in\mT$ be smooth of ramification type $R=2,1^2$.
 The line $\ell$ meets exactly 16 lines on $\XXf$  if and only if the coefficient of $x_3^4$ in $g$  equals $4a^3c^3/27$.\end{cor}

\begin{proof}
By \cite[\S 4]{ramsschuett} the ramified $I_1$ fiber degenerating to Kodaira type
$I_2$ is the only degeneration which increases the 
number of lines on $X$ met by $\ell$ without the ramification type changing.
Hence the corollary follows from Lemma \ref{lem:ftype112} (b).
\end{proof}

Cor.~\ref{prop:15,16}  gives the announced counterexample to Segre's Claim~\ref{eq-numberoflines}
for ramification type $R=2,1^2$. 
Now we are in the position to study  a general element of $\mT$. 
The question when a smooth quartic $\XXf$
with  a line of the second kind with ramification type $R=2,1^2$ 
can be obtained by Construction \ref{obs-segretrick} will be answered by Cor.~\ref{cor:211segre}~(a).

\begin{prop} \label{prop:geninmT}
\label{prop:S_5}
Let the quartic $X$ be a general member of the family $\mT$. 

%
%

{\sl (a)} The quartic $\XXf$ is not given by Construction~\ref{obs-segretrick}.

{\sl (b)} The surface $\Ruledeight$ consists of three planes and an irreducible quintic.
\end{prop}
\begin{proof} 

%
{\sl (a)}  The line $\ell$ meets a fiber of type $I_1$   in only one point by Lemma~\ref{lem:ftype112}~(a).
The claim follows directly from  Lemma~\ref{lem:ramIV}.

{\sl (b)} We follow the proof of Lemma~\ref{lemma-ruledoctic} to find the surface $\Ruledeight$. By direct computation
$\Ruledeight$ consists of the tangent planes  $\mbox{T}_{P_1}\XXf$, $\ldots$, 
 $\mbox{T}_{P_3}\XXf$ and a  quintic $\Ruledfive$. 
In particular, by direct computation, $\Ruledfive$ contains none of  the planes $\mbox{T}_{P_1}\XXf$, $\ldots$, 
 $\mbox{T}_{P_3}\XXf$. 
 By Lemma~\ref{lem:planeS8} the quintic  $\Ruledfive$ is irreducible.
\end{proof}

In order to give an intrinsic characterization of  
quartics with lines of the second kind with ramification type $R=2,1^2$ 
that can be obtained by Segre's
construction,
we will need the following lemma.

\begin{lemm} \label{lemm:15,16b}
Let $X\in\mT$ be smooth of ramification type $R=2,1^2$
and  let   the fiber $F_{P_1}$ be of Kodaira type $IV$
(i.e.  $\ell$  runs through the singular point of a type $IV$ fiber  of the fibration \eqref{eq:fibration}).
Then
 \begin{enumerate}
\item[(a)] The line $\ell$ is met by exactly 15 other lines on $\XXf$.
\item[(b)] The 15 lines that meet $\ell$  form five triplets of coplanar lines.
None of them is contained in any of the planes $\mbox{T}_{P_2}\XXf$, $\mbox{T}_{P_3}\XXf$. 
\item[(c)] The fibration \eqref{eq:fibration} has no  fibers of Kodaira type $I_2$.
\end{enumerate}
\end{lemm}

\begin{proof} (a), (b)  follow directly from the proof of Lemma~\ref{lem:ftype112}.

{\sl (c)} By \cite[Lemma~3.2]{ramsschuett} only the fiber $F_{P_1}$ can be of type $I_2$,
but it isn't by assumption.
\end{proof}

Now we are in the position to 
give the desired coordinate-free characterization
of quartics with lines of the second kind with ramification type $R=2,1^2$ 
which can be obtained by Segre's
construction. We maintain the notation \eqref{eq:rl211}.

\begin{cor} \label{cor:211segre}
\label{cor:6.7}
Let $\XXf \subset \PP^3$ be a smooth quartic
and let $\ell\subset \XXf$ be a line of the second kind 
with ramification type $R=2,1^2$. 

(a) The quartic $\XXf$ is given by Construction~\ref{obs-segretrick} iff 
the line $\ell$ runs through the  singular point of a type $IV$ fiber of   the fibration  \eqref{eq:fibration}.

(b) If the equivalent conditions of (a) hold, then                 
\begin{itemize}
\item the surface $\Ruledeight$ defined by $(\XXf, \ell)$ is a septic that consists of the tangent planes $\mbox{T}_{P_1}\XXf$, $\ldots$
 $\mbox{T}_{P_3}\XXf$ and an (irreducible) ruled quartic $\Ruledfour$;
\item the line $\ell$ is met by exactly five triplets of coplanar lines on $\XXf$, that are contained in the plane 
$\mbox{T}_{P_1}\XXf$,  and four other planes $L_1$, $\ldots$, $L_4$;
\item the decomposition \eqref{eq-segretrick} is unique and coincides with \eqref{eq:segre1111}; 
\item the fibration \eqref{eq:fibration} is of type $IV \oplus 4I_3 \oplus 4I_1 \oplus 2II$ (resp. $5 IV \oplus 2II$) iff
$\sing(\Ruledfour)$ is a twisted cubic (resp. a line). 
\end{itemize} 
\end{cor}

\begin{proof}(a) 
The implication ($\Rightarrow$) follows directly from  Lemma~\ref{lem:ramIV}.

($\Leftarrow$) 
By Lemma \ref{lem:2,1^2} we can assume that $\XXf \in \mT$. Suppose that the fiber $F_{P_1}$ is of type $IV$.
By Lemma~\ref{lem:ftype112}~(c) the quartic $\XXf$ is given by the equation of Lemma~\ref{lem:2,1^2} with $a = 0$.
We put $a=0 $ in 
the quintic $S_5$ from (the proof of) Proposition \ref{prop:S_5} to see that 
$\Ruledeight$ consists of the tangent planes  $\mbox{T}_{P_1}\XXf$, $\ldots$, 
 $\mbox{T}_{P_3}\XXf$ and a  quartic $\Ruledfour$. 
In particular, by direct computation, for $i= 1, 2,3$
\begin{equation} \label{eq:intplanes21}
\mbox{ the intersection of  }  \Ruledfour \mbox{ with the plane } \mbox{T}_{P_i}\XXf \mbox{ consists of lines},  
\end{equation}
so $\Ruledfour$ contains none of  the planes $\mbox{T}_{P_1}\XXf$, $\ldots$, 
 $\mbox{T}_{P_3}\XXf$.
The assumption \eqref{eq:decS4tr}  of Lemma~\ref{lem:S4tr} is fulfilled. By Lemma~\ref{lemm:15,16b}~(b) there exist four planes $L_1, \ldots L_4$, each of which contains 
four lines on $\XXf$ that meet $\ell$. Lemma~\ref{lem:S4tr} completes the proof of (a).

(b) Let $\Pi \subset \PP^3$ be a plane that contains $\ell$.  By definition of $\Ruledeight$ the intersection
$\Ruledfour \cap \Pi$
consists of lines, provided $\Pi \neq \mbox{T}_{P_i}\XXf$  where $i= 1, 2,3$.
Thus \eqref{eq:intplanes21} yields that   $\Ruledfour$ is a union of a family of lines each of which meets $\ell$. Lemma~\ref{lemma-linewayfromsingularities} implies that $\Ruledfour$ is an 
(irreducible) ruled quartic. Thus Lemma~\ref{lemma-splitplanesegre} gives the first claim of (b),
whereas  the second claim results immediately from Lemma~\ref{lemm:15,16b}. 

To prove the  uniqueness suppose that $\XXf$ is given by \eqref{eq-segretrick}.
By Lemma~\ref{lemma-linewayfromsingularities}, the quartic $\Ruledfourf$
 is a ruled surface, $\ell$ is its directrix and
$\sing(\Ruledfourf)$ is either a twisted cubic or a line of triple points.
Moreover, we have  
the equality ($\ast$) of \eqref{eq:tllhsect0}, so the plane  $\TLL_j$   intersects the quartic $\XXf$ along four lines.
Since $\ell$ contains no singularities of $\Ruledfourf$, the curve
$\TLL_1 \cap \XXf$ has at least one singularity away from $\ell$. Therefore 
\begin{equation} \label{eq:nottan}
\TLL_1 \neq  \mbox{T}_{P_1}\XXf . 
\end{equation}
Furthermore, the quartic $\XXf$ is smooth and $\Ruledfourf$ is singular along a curve, so $\TLL_{j_1} \neq \TLL_{j_2}$ for $j_1 \neq j_2$. 
Thus Lemma~\ref{lemm:15,16b}~(b) implies the equality
\begin{equation} \label{eq:4planesequal}
\{ \TLL_1,  \ldots, \TLL_4 \} = \{L_1,  \ldots, L_4\}.
\end{equation}
The uniqueness follows directly from Lemma~\ref{lem:4eqplanes}.

To show the last claim recall that \eqref{eq:fibration} has two type $II$ fibers (see Lemma~\ref{lem:ftype112}~(b)) and none of type $I_2$  by Lemma~\ref{lemm:15,16b}~(c).
It has one type $IV$ fiber contained in $ \mbox{T}_{P_1}\XXf$ by assumption. 

Assume that $\sing(\Ruledfour)$ is a twisted cubic. Then \eqref{eq:fibration} has four $I_3$ fibers by Cor.~\ref{cor:singfib}. 
The existence of four singular fibers
of Kodaira type $I_1$ follows from \cite[Lemma~4.2]{ramsschuett}. 

Suppose that  $\sing(\Ruledfour)$ is a line. We obtain  four type $IV$ fibers from Cor.~\ref{cor:singfib}. 
In both cases  Euler number computation shows that we found all singular fibers.
\end{proof}

\section{Ramification type $2^2$}
\label{s:corr}

Let  $X\subset\PP^3$ be a smooth quartic  that 
contains a line $\ell$ of the second kind with ramification type $R=2^2$ and let 
\begin{equation} \label{eq:rl22}
\mathcal R  = 2  P_1 + 2 P_2.
\end{equation}
By \cite[Lemma 4.4]{ramsschuett}
the quartic $\XXf$  
is projectively equivalent to a member of the family $\mZ$ given by the polynomials
\begin{eqnarray}
\label{eq:Z}
\mZ: \;\; x_3x_1^3+x_4x_2^3+x_1x_2q(x_3,x_4)+g(x_3,x_4), 
\end{eqnarray}
where $q, g  \in \KK[x_3, x_4]$
are  homogeneous polynomials of degree $2$, resp.~$4$. Here the line $\ell$ becomes $\mbox{V}(x_3, x_4)$ and
   $P_1=(1 :0 : 0: 0), P_2=(0: 1: 0 :0)$. Moreover,  \cite[Lemma~4.5]{ramsschuett}
yields that
\begin{equation} \label{eq:condIV}
\mbox{both fibers } F_{P_1}, F_{P_2} \mbox { are of Kodaira type } IV \quad \mbox{iff} \quad  q=\gamma x_3x_4 \mbox{ for some } \gamma\in\KK.
\end{equation}


In this section we  investigate the quartics in the family $\mZ$. 
 At first we study general elements of $\mZ$ (see Prop.~\ref{prop-generic-in-mZ}). 
Quartics that can be obtained by Construction \ref{obs-segretrick} are discussed in  Cor.~\ref{cor:22segre}~(a).

\begin{prop} \label{prop-generic-in-mZ}
Let the quartic $X$ be a general member of the family $\mZ$.

{\sl (a) } $\XXf$ is smooth and contains exactly $18$ lines $\ell' \neq \ell$ that meet the line of the second kind $\ell$. 

{\sl (b)}  The quartic $\XXf$ cannot be obtained from  Construction~\ref{obs-segretrick}.

{\sl (c)}  The surface $\Ruledeight$ consists of two planes and an irreducible sextic.

\end{prop}
\begin{proof} {\sl (a)}
We consider the elliptic fibration $\pi$ 
induced by the line $\ell\subset X \in \mZ$ (see \eqref{eq:fibration}).
Generically, $\pi$ has six singular fibers of Kodaira type $I_1$
located at $0, \infty$ and at the zeroes of $g$,
and 6 fibers of Kodaira type $I_3$ at the zeroes of
$q^3+27x_3x_4g$.
In other words, the discriminant $\Delta$ of the generic fiber is
\[
\Delta = x_3x_4g(q^3+27x_3x_4g)^3.
\]
In particular, for general $\XXf \in \mZ$ the fibration \eqref{eq:fibration} is of type
\begin{equation} \label{eq:gf2^2}
6 I_3 \oplus 6 I_1 .
\end{equation}

Recall that singularities of $X$ either give rise to some of the reducible fibers
in the first place
(if the generic member of $\mZ$ were not to be  smooth),
or to a degeneration of the fibration.
That is, either there are  further or worse singular fibers,
or the degeneration leaves the class of K3 surfaces.
Presently one can easily check for some specific member $X\in\mZ$, 
for instance for char$(\KK)\neq 2,3,5,7$ at 
\[
q=3(2x_3^2-x_3x_4+x_4^2), \;\; g=4(20 x_3^4+5 x_4^4-18 x_3^3 x_4-4 x_3^2 x_4^2-9 x_3 x_4^3)/3
\]
that the six planes in $\PP^3$ at the zeroes of $q^3+27x_3x_4g$
split into $\ell$ and three other lines when intersected with $X$.
In particular, this proves that the general member of $\mZ$ is smooth
and contains exactly $6\times 3 = 18$ lines meeting $\ell$.

{\sl (b)} The claim follows directly from \eqref{eq:gf2^2} and Lemma~\ref{lem:ramIV}.

{\sl (c)} By Lemma \ref{lemma-splitplanesegre} the surface $\Ruledeight$ contains the planes tangent to $\XXf$ in $P_1, P_2$. 
The residual surface will be shown to be an irreducible sextic by Lemma~\ref{lem:S_6}.
\end{proof}

To  complete the proof of Prop.~\ref{prop-generic-in-mZ} we need the following lemma.
\begin{lemm} 
\label{lem:S_6}
For a smooth quartic $X\in\mZ$, the surface $S_8$ contains
an irreducible sextic hypersurface 
unless $x_3$ or $x_4$ divides $q$.
\end{lemm}

\begin{proof} 
Following  the proof of Lemma~\ref{lemma-ruledoctic},
we arrive at the sextic polynomial
\begin{eqnarray}
\label{eq:h}
h:=27x_1^3x_3^2x_4+27x_2^3x_3x_4^2+27x_1x_2x_3x_4q-q^3.
\end{eqnarray}
Note that $x_3$ or $x_4$ divides $h$ if and only if the coordinate in question divides $q$.
Otherwise, we can multiply $h$ by $x_3x_4^2$ and then substitute $x_1$ by $x_1/(x_3x_4)$,
so that $h$ is equivalent to the following polynomial:
\[
h':=27x_1^3+27x_2^3x_3^2x_4^4+27x_1x_2x_3x_4^2q-x_3x_4^2q^3.
\]
Regarding $h'$ as an element in $\KK[x_2,x_3,x_4][x_1]$,
the Eisenstein criterion applied to the prime ideal $(x_3)$ shows that
$h'$ is irreducible if $x_3\nmid q$.
\end{proof}

Prop.~\ref{prop-generic-in-mZ}
shows that neither 
Claim~\ref{eq-splittingquartic} nor
Claim~\ref{eq-generalequation} is true in general for ramification type $R=2^2$.  
In particular, \emph{it is the erroneous Claim~\ref{eq-splittingquartic} 
that brings about the crucial false statements about configurations 
of lines  made by Segre in 
\cite{Segre}}
which we corrected in \cite{ramsschuett}.

We shall now elaborate the family $\mZ$ a little further to study the quartics given by Construction~\ref{obs-segretrick}.
For a specific member $X\in \mZ$ to be smooth,
it suffices to verify the following criterion:
\begin{lemm} 
\label{lem:smooth}
A quartic $X\in \mZ$ is smooth if and only if
\begin{itemize}
\item
$\Delta$ has still 6 single and 6 triple roots with the possible extension that
\item
$\Delta$ may have double roots at $0, \infty$ (i.e.~square factors of $x_3, x_4$) and
\item
$\Delta$ may have fourfold roots at the zeroes of $q^3+27x_3x_4g$.
\end{itemize}
\end{lemm}
 
\begin{proof}
This is mostly \cite[Lemma~4.5]{ramsschuett}, but we give some details here for completeness.
It is trivial, but crucial to note that if $X$ is smooth,
the fibers of $\pi$ are the residual cubics,
so they consist of at most 3 components.
By the classification of Kodaira \cite{K} and Tate \cite{Tate},
this allows for six different types of singular fibers, listed below with corresponding vanishing order $v$ of $\Delta$:
$$
\begin{array}{r|cccccc}
\text{fiber type} & I_1 & I_2 & I_3 & II & III & IV\\
v(\Delta) & 1 & 2 & 3 & 2 & 3 & 4
\end{array}
$$

There is a hidden 3-torsion structure on the generic fiber of $\pi$
as explored in \cite[\S 3,4]{ramsschuett}.
This severely limits the possible singular fibers;
for instance, there cannot be any fiber of Kodaira type $II$ or $III$,
and  the $I_3$ fibers cannot degenerate to type $I_4$ or $I_5$
(compare \cite[Lemma 3.2 and Prop.~4.1]{ramsschuett}). 
Hence, for instance, if $\Delta$ has a fourfold root at a zero of $q^3+27x_3x_4g$,
then the degenerate fiber has automatically Kodaira type $IV$;
in particular the degeneration is smooth.
Similarly, a double root at $0$ or $\infty$ is easily seen to correspond
to the hyperplane $x_3=0$ or $x_4=0$ splitting off a conic and a line other than $\ell$
when intersected with $X$,
so this again gives smooth degenerations.

On the contrary, a double root of $\Delta$ outside $0,\infty$
necessarily occurs at the root of a square factor of $g$.
But then this results in a singularity of $S$ at the double root with $x_1=x_2=0$.
Along the same lines,
a triple root of $\Delta$ at $0$ or $\infty$ implies that $x_3^2$ or $x_4^2$ divides $g$,
causing a singularity on $X$ again.
\end{proof}

%

As recorded in the previous paragraphs,
the degenerations of $X \in \mZ$ where $x_3$ or/and $x_4$ divide $g$
cause the quartic $X$ to contain 1 or 2 additional lines meeting $\ell$
(see \cite[Example~6.9]{ramsschuett} for an explicit example).
This implies that \emph{ Claims~\ref{eq-numberoflines},~\ref{eq-minibound} which are true generically on $\mZ$,
are false for  certain specific examples of
ramification type $R=2^2$.}


We maintain the notation \eqref{eq:rl22} and collect  extra information on the lines $\ell' \neq \ell$ that meet $\ell$
in the lemma below.
\begin{lemm} 
\label{lem:varia}
 Let $\ell$ be a line of the second kind 
with ramification type $R=2^2$ on a smooth quartic $\XXf \subset \PP^3$. Moreover, we assume that 
  the fibers $F_{P_1}$, $F_{P_2}$ are  of Kodaira type $IV$.

\begin{enumerate}

\item[(a)] The fibration \eqref{eq:fibration} has no singular fibers of type $I_2$.

\item[(b)]
The line $\ell$ is met by exactly $18$ lines $\ell' \neq \ell$ on $\XXf$. They form six triplets of coplanar lines.


 \end{enumerate}
\end{lemm}
\begin{proof}
{\sl (a)}  By \cite[Lemma~3.1]{ramsschuett}, a type $I_2$ fiber of the fibration \eqref{eq:fibration} 
must be either $F_{P_1}$ 
or  $F_{P_2}$.

{\sl (b)} By \cite[Prop.~4.1]{ramsschuett} (see also study of the case $R=2^2$ in the proof of  \cite[Lemma~4.2]{ramsschuett})
the line $\ell$ is always met by six triplets of coplanar lines on $\XXf$. 
By \cite[Lemma~4.5]{ramsschuett}, if  $\ell$ is met by either $19$ or $20$ lines $\ell' \neq \ell$, then each  line $\ell'$ that does not belong to one of the six triplets
   is a component of a fiber of type $I_2$ of the fibration \eqref{eq:fibration}.  
Thus (a) completes the proof.  
%
\end{proof}

Now we are in the position to 
give an intrinsic characterization of  
quartics with lines of the second kind with ramification type $R=2^2$ that can be obtained by Segre's
construction.
\begin{cor} \label{cor:22segre}
Let $\XXf \subset \PP^3$ be a smooth quartic
and let $\ell\subset \XXf$ be a line of the second kind 
with ramification type $R=2^2$.

(a) The quartic $\XXf$ is given by Segre's construction
\eqref{eq-segretrick} iff the line $\ell$ runs through the  singular points of two singular fibers of Kodaira type $IV$ 
of the fibration \eqref{eq:fibration}.

(b) If the assumption  of (a) is satisfied, then  
\begin{itemize}
\item the surface $\Ruledeight$ defined by $(\XXf, \ell)$ is a sextic that consists of the tangent planes $\mbox{T}_{P_1}\XXf$, 
 $\mbox{T}_{P_2}\XXf$ and an (irreducible) ruled quartic $\Ruledfour$,
\item the line $\ell$ is met by exactly six triplets of coplanar lines on $\XXf$ that are contained in the planes 
$\mbox{T}_{P_1}\XXf$, 
 $\mbox{T}_{P_2}\XXf$, and in four other planes $L_1$, $\ldots$, $L_4$,
\item the decomposition \eqref{eq-segretrick} is unique and coincides with \eqref{eq:segre1111},
\item the fibration \eqref{eq:fibration} is of type $2IV \oplus 4I_3 \oplus 4I_1$ (resp. $6 IV$) iff
$\sing(\Ruledfour)$ is a twisted cubic (resp. a line). 
\end{itemize} 
\end{cor}
\begin{proof} (a) 
The implication ($\Rightarrow$) is a  consequence of Lemma~\ref{lem:ramIV}.

($\Leftarrow$) We can assume that $\XXf \in \mZ$ and  the fibers $F_{P_1}$, $F_{P_2}$ are of type $IV$.
By \eqref{eq:condIV}
 the quartic $\XXf$ is given by the equation \eqref{eq:Z} with  $q=\gamma x_3x_4$ for some $\gamma\in\KK$.
At first we follow the proof of   Lemma~\ref{lemma-ruledoctic} to
compute the equation of the surface $\Ruledeight$ defined by the pair
$(\XXf, \ell)$. 
To this end,
we plug $q=\gamma x_3x_4$ into \eqref{eq:h} to see that 
the sextic given by \eqref{eq:h} contains the planes $\mbox{V}(x_3)$,  $\mbox{V}(x_4)$.
Canceling out the factor $x_3 x_4$, we derive 
 the equation of the quartic $\Ruledfour$ residual to the above planes 
in the surface $\Ruledeight$ (cf.~Remark \ref{rem:eq-S_4}).
By direct computation 
\begin{equation} \label{eq:2lines}
\Ruledfour \mbox{ meets  each plane } \mbox{V}(x_3) \mbox{ and } \mbox{V}(x_4) \mbox{ along a union of two lines.}
\end{equation}
In particular, neither $\mbox{T}_{P_1}\XXf = \mbox{V}(x_3)$ nor $\mbox{T}_{P_2}\XXf = \mbox{V}(x_4)$ are components of the hypersurface $\Ruledfour$,
so $\Ruledfour$ is irreducible by Lemma \ref{lem:planeS8},
and we have
\begin{equation} \label{eq:decS8case22}
\Ruledeight = \Ruledfour \cup \mbox{T}_{P_1}\XXf \cup \mbox{T}_{P_2}\XXf .
\end{equation}
Lemma~\ref{lem:varia}~(b) and \eqref{eq:decS8case22} imply that the assumptions of Lemma~\ref{lem:S4tr} are fulfilled. Thus $X$ is given by   \eqref{eq:segre1111}.
Furthermore,  
by definition of $\Ruledeight$ and \eqref{eq:2lines}, the quartic $\Ruledfour$ is a union of a family of lines each of which meets $\ell$. Thus it is ruled (in particular, 
it is irreducible -- see Lemma~\ref{lemma-linewayfromsingularities}).  This completes the proof of (a).

(b) The first claim follows from \eqref{eq:decS8case22}. The second claim results immediately from Lemma~\ref{lem:varia}. 

To prove the next part of (b) (i.e. the uniqueness) we can follow almost verbatim the proof in case of ramification $2,1^2$ (see Cor.~\ref{cor:211segre}).
We suppose that $\XXf$ is given by \eqref{eq-segretrick} and show that (compare \eqref{eq:nottan})
$$
\TLL_1 \neq  \mbox{T}_{P_1}\XXf, \mbox{T}_{P_2}\XXf . 
$$
Thus we arrive at the equality \eqref{eq:4planesequal} and apply Lemma~\ref{lem:4eqplanes}.

Finally, the fibration \eqref{eq:fibration}  is assumed to have two fibers of type $IV$. It has  
four fibers of type $I_3$ (resp. four extra fibers  of type $IV$) iff $\sing(\Ruledfourf)$ is a twisted cubic (resp. a line of triple points) by Cor.~\ref{cor:singfib}. 
Lemma~\ref{lem:varia}~(a) combined with \cite[Lemma~4.2]{ramsschuett} completes the proof.
\end{proof} 

\begin{rem}
\label{rem:eq-S_4}
In case of a quartic in the family $\mZ$,
the condition of Corollary \ref{cor:22segre} (a) translates to $q =\gamma x_3x_4$ for some  $\gamma\in\KK$.
Then the ruled quartic $\Ruledfour$ is given by the polynomial
$$ 
S_4:\;\;\; x_3x_1^3+x_4x_2^3+  \gamma  x_1x_2x_3x_4 -\gamma^3 x_3^2x_4^2/27.
$$
As one can check, the singular locus of $\Ruledfour$ is a twisted cubic (resp. a line) iff $\gamma \neq 0$ (resp. $\gamma=0$). 
\end{rem}

\section{Proof of Theorem \ref{thm}}
\label{s:thm}

Up to projective equivalence, the families $\mT, \mZ$ are 6-dimensional as one can still rescale coordinates
while preserving the two normalisations of Lemma \ref{lem:2,1^2} resp.~\eqref{eq:Z} 
to eliminate two of the 8 parameters.
As  the singular fibers generally differ,
we infer that neither family is a subfamily of the other.
By inspection of the singular fibers,
the same can be said about the ramification type $R=1^4$,
but we omit the details here since the explicit equations become too involved.
This explains why proving statements about ramification types $R=2,1^2$ and $R=2^2$
by degenerating from $R=1^4$,  as apparently believed to work by Segre for Claims~\ref{eq-splittingquartic} -- \ref{eq-minibound},
is bound to fail.

\subsubsection*{Proof of Theorem  \ref{thm} (a)}
For the ramification type $R=1^4$ it was proved in Proposition \ref{prop:5.2}
that any such smooth quartic takes the shape of \eqref{eq:segint}.

For ramification type $R=2,1^2$ it was proved in Corollary \ref{cor:6.7}
that any smooth quartic with a line $\ell$ of the second kind
of ramification type $R$ takes the shape of \eqref{eq:segint}
if and only if the singular fiber of \eqref{eq:fibration} met by $\ell$ in a single point
(which generically has Kodaira type $I_1$)
degenerates to type $IV$,
that is, iff \eqref{eq:segcond} holds.
In terms of the explicit family $\mT$,
the codimension 1 condition $a=0$ is given in Lemma \ref{lem:6.3} (b).
This proves the extra claim of Theorem \ref{thm} for this ramification type.


For ramification type $R=2^2$ it was proved in Corollary \ref{cor:22segre} (a)
that any smooth quartic with a line $\ell$ of the second kind
of ramification type $R$ takes the shape of \eqref{eq:segint}
if and only if  both singular fibers of \eqref{eq:fibration} met by $\ell$ in a single point
(which generically have Kodaira type $I_1$)
degenerate to type $IV$.
Euiqvalently, the fibration \eqref{eq:fibration} satisfies \eqref{eq:segcond}.
As for the extra claim for ramification type $R=2^2$,
the codimension 2 conditions are given in terms of the 
 family $\mZ$ in Remark \ref{rem:eq-S_4}.

\subsubsection*{Proof of Theorem  \ref{thm} (b)}

This is exactly Lemma \ref{lem:2,1^2} resp.~\cite[Lemma 4.4]{ramsschuett}.

\subsection*{Acknowledgement}

We thank Wolf Barth and Jaap Top for helpful discussions.


\begin{thebibliography}{99}


 

\bibitem{boissieresarti} Boissi\'ere, S., Sarti, A.: \emph{Counting lines on surfaces.} Ann. Sc. Norm. Super. Pisa, Cl. Sci. {\bf 6} (2007), 39--52.

%


%

\bibitem{harris-caporaso} Caporaso~L., Harris~J., Mazur~B.: \emph{How many rational points can a curve have.} 
in The Moduli Space of Curves (R.~Dijkgraaf, C~Faber, G.~van~der~Geer eds.),
Progress in Math. 129, Birkh\"auser Verlag, 1995, 13--31.


\bibitem{harris-tschinkel}  Harris, J.; Tschinkel, Y.:  \emph{Rational points on quartics.} Duke Math. J. {\bf 104} (2000), no. 3, 477--500. 


\bibitem{K} Kodaira, K.:
\emph{On compact analytic surfaces I-III}.
Ann.~of Math., {\bf 71} (1960), 111--152;
{\bf 77} (1963), 563--626; {\bf 78} (1963), 1--40.


\bibitem{puttop} Polo-Blanco, I., van der Put, M., Top, J.: \emph{Ruled quartic surfaces, models and classification.} Geom. Dedicata {\bf 150} (2011), 151--180.

\bibitem{ramsdisjoint} Rams, S.: 
\emph{Projective surfaces with many skew lines.} Proc. Amer. Math. Soc. {\bf 133} (2005), 11--13.



\bibitem{ramsschuett} Rams, S., Sch\"utt, M.: \emph{64 lines on smooth quartic surfaces.} Preprint available at arXiv:math/1212.3511v1, 2012.


%


%

\bibitem{Segre} Segre, B.: \emph{The maximum number of lines lying on a quartic surface},  Quart. J. Math., Oxford Ser. {\bf 14} (1943), 86--96.












\bibitem{Tate} Tate, J.: {\it Algorithm for determining the type
of a singular fiber in an elliptic pencil}.
{\it Modular
functions of one variable IV} (Antwerpen 1972), Lect.~Notes in Math.~{\bf 476}
(1975), 33--52.


\bibitem{Voloch}
Voloch, F.:
\emph{Surfaces in $\PP^3$ over finite fields},
Topics in Algebraic and Noncommutative Geometry,
Contemp.~Math.~{\bf 324} (2003), 219--226.
















\end{thebibliography}
\end{document}